\documentclass[12pt]{amsart}
\usepackage{amsmath}
\usepackage{amsfonts}
\usepackage{amssymb}
\usepackage{setspace}
\usepackage{moreverb}
\usepackage{euscript,bbold}
\usepackage[latin1]{inputenc}
\usepackage[T1]{fontenc}
\usepackage{amssymb}
\usepackage{tikz}
\usepackage{tikz-3dplot}
\usepackage{wrapfig}
\usepackage{pgfplots}
\usetikzlibrary{patterns}
\usepackage{color}
\usepackage[all]{xy}
\usepackage{xfrac}
\usepackage{relsize}
\usepackage[colorlinks=true,pagebackref,hyperindex,citecolor=blue,linkcolor=red]{hyperref}
\usepackage[top=1in, bottom=1in, left=1in, right=1in]{geometry}
\usepackage{mathrsfs}

\usetikzlibrary{shapes}

\pgfplotsset{compat=1.12}
\newtheorem{theoremx}{Theorem}

\newtheorem{theorem}{Theorem}[section]

\newtheorem{corollary}[theorem]{Corollary}
\newtheorem{lemma}[theorem]{Lemma}
\newtheorem{proposition}[theorem]{Proposition}

\theoremstyle{definition}
\newtheorem{definition}[theorem]{Definition}

\newtheorem{notation}[theorem]{Notation}

\newtheorem{example}[theorem]{Example}

\newtheorem{conjecture}[theorem]{Conjecture}

\newtheorem{remark}[theorem]{Remark}
\numberwithin{equation}{subsection}

\DeclareMathOperator{\fpt}{fpt}

\newcommand{\lct}{{\operatorname{lct}}}

\newcommand{\set}[1]{\left\{ #1 \right\}}

\newcommand{\CC}{\mathbb{C}}
\newcommand{\ZZ}{\mathbb{Z}}
\newcommand{\QQ}{\mathbb{Q}}
\newcommand{\FF}{\mathbb{F}}
\newcommand{\RR}{\mathbb{R}}

\newcommand{\m}{\mathfrak{m}}
\newcommand{\fa}{\mathfrak{a}}

\newcommand{\fb}{\mathfrak{b}}
\newcommand{\SP}{\mathcal{P}}
\newcommand{\N}{\mathcal{N}}
\newcommand{\Bf}{\boldsymbol{f}}

\newcommand{\Card}{\operatorname{Card}}
\newcommand{\E}{\mathcal{E}}
\newcommand{\tE}{\operatorname{\widehat{\E}}}

\newcommand{\Supp}{\operatorname{Supp}}

\newcommand{\Vol}{\operatorname{Vol}}
\newcommand{\V}{\operatorname{V}}

\allowdisplaybreaks

\begin{document}
	
\title{$F$-pure thresholds and $F$-Volumes of some non-principal ideals}
	
\author[W. Badilla-C\'espedes]{W\'agner Badilla-C\'espedes$^{1}$}
\address{Centro de Ciencias Matemáticas, UNAM, Campus Morelia, Morelia, Michoacán, México.}
\email{wagner@matmor.unam.mx}

\author[E. Le\'on-Cardenal]{Edwin Le\'on-Cardenal$^{2}$}
\address{Departamento de Matemáticas y Computación, Universidad de La Rioja\\
	C. Madre de Dios, 53, 26006, Logro\~no, Spain.}
\email{edwin.leon@unirioja.es}
\thanks{$^{1}$ The first author was supported by UNAM Posdoctoral Program (POSDOC) and by SNII (CONAHCYT, México). $^{2}$ The second author is partially supported by PID2020-114750GB-C31, funded by MCIN/AEI/10.13039/501100011033 and the Departamento de Ciencia, Universidad y Sociedad del Conocimiento of the	Gobierno de Arag\'on (Grupo de referencia ``$\acute{A}$lgebra y Geometr\'ia'') E22\_20R. The second author was also partially supported by the European Union NextGenerationEU/PRTR via Maria Zambrano's Program. Both authors were partially supported by CONAHCYT project CF-2023-G33.}


\begin{abstract}
	We compute the $F$-pure threshold of some non-principal ideals which satisfy a geometric generic condition about their Newton polyhedron. We also contribute some evidence in favor of the conjectured equality between the $F$-pure threshold and the log canonical threshold of ideals for infinitely many primes $p$. These results are obtained by generalizing the theory of splitting polytopes to the case of ideals. As an application of our results we obtain geometric lower bounds for the recently introduced $F$-volume of a collection of ideals.
\end{abstract}
	
\maketitle
	
\section{Introduction}
The problem of quantifying singularities of an algebraic variety $X$ at a given point $x$  can be approached in several ways, but the first thing to consider is the characteristic of the ground field. Over fields of characteristic zero it is common to use the log canonical threshold of $X$ at the point $x$, $\lct_x(X)$, which is a numerical invariant that can be defined analytically (via integration when available, e.g. over $\RR$ or $\CC$) or algebro-geometrically (via resolution of singularities). Over fields of prime characteristic, the corresponding numerical invariant is called the $F$-pure threshold of $X$ at the point $x$; it is denoted by $\fpt_x(X)$ and its definition uses the Frobenius map. While defined in an entirely different way, it turns out that both invariants share several properties and moreover, if $X_p$ denotes the reduction of the
variety $X$ to characteristic $p$, it is known that $\fpt_x(X_p)$ approaches $\lct_x(X)$ as long as $p$ increases. Furthermore, there is a largely open conjecture asserting that $\fpt_x(X_p)=\lct_x(X)$ for infinitely many primes $p$.

For definiteness, consider a polynomial ring $R$ over a field $\mathbb{k}$ of positive characteristic $p$. The $F$-pure threshold (at the origin) of an ideal $\fa\subseteq\m$ is defined as the limit when $e\to \infty$ of $p^{-e}(\max \{c \in \mathbb{Z}_{\geq 0} \mid \mathfrak{a}^c \not \subseteq \m^{[p^{e}]}\})$. Here $\m$ denotes the homogeneous maximal ideal of $R$ and $\m^{[p^{e}]}$ is the $p^{e}$-th Frobenius power of $\m$. This definition was given by 
Takagi and Watanabe \cite{TW2004} and since those days many properties and connections with other objects have been elucidated as it is summarized in the surveys \cite{BFS2013,TaWa2018} and the references therein. However it is acknowledged by the community that the task of computing $\fpt_x(X)$ is hard in general and few examples of explicit calculations are known, among them \cite{ShTa,HDbinomial,MiSiVa,BhSi15,HDpolynomials,HNBW2016,Gil18,Mu18,Gil22,Tri,GVJVNB,SmVra}. In fact, the first computational routine for testing some examples was developed just a few years ago~\cite{MacFrobThre}.
 
Another serious difficulty that emerges when it comes to computing explicit examples of $F$-pure thresholds is that the existing methods are designed mostly for the case of principal ideals; the situation in this regard is similar to the one for test ideals \cite{ScTu2014}. To the best of our knowledge the more general results in this direction are due to Shibuta and Takagi \cite{ShTa}, who present several explicit calculations of 
$F$-pure thresholds of binomial ideals by using linear programming. The authors generalize there the notable case of monomial ideals that can be deduced from the work of Howald on multiplier ideals \cite{How}. In this work we attack the case of non-principal ideals by extending the techniques of splitting polytopes developed by Hern\'andez \cite{HDbinomial,HDpolynomials}. The splitting polytope $\SP_f$ of a polynomial $f\in R$ is a compact convex polytope that resembles the more classical notion of the Newton polyhedron. In fact there are several useful connections between these two objects, allowing one to formulate some algebro-geometric conditions of the hypersurface defined by $f$ in terms of $\SP_{f}$. By using the definition of the Newton polyhedron of a polynomial mapping $\Bf=(f_1,\ldots,f_t)$ \cite{VeZu,LeVeZu} we define and study a splitting polytope $\SP_{\Bf}$ for an ideal $\fa=(f_1,\ldots,f_t)$. The term ideal $\fa^\circ$ of a non necessarilly principal ideal is a monomial ideal that has been used to study for example colength of ideals \cite{BiFuSa}. In our first main result we determine some geometric conditions under which  $\fpt(\fa)$ is equal to $\fpt(\fa^\circ)$. In this way we provide an extensive list of new examples of $F$-pure thresholds for non-principal ideals that moreover can be computed explicitly, see Example~\ref{Ex:Fthrsholds} and Proposition~\ref{Prop:Properties ftt ideals}(2). Given a prime number $p$ and $\rho\in(0,1]$,  $\rho^{(e)}$ denotes the $e$-th term in the $p$-expansion of $\rho$, while $\langle \rho \rangle_e$ is the $e$-th truncation of such expansion. Moreover, for $\rho\in\RR^l$, the symbol $|\rho|$ denotes the sum of all the entries of $\rho$. 

\begin{theoremx}[{Theorem \ref{Thm:F-pure-threshold-polytope}}]\label{Thm: first main theorem}
		Let $\fa \subseteq \m$ be a nonzero ideal with term ideal $\fa^\circ$, and take a list $\Bf=(f_1,\ldots,f_t)$  of minimal generators of $\fa$. Suppose that $\SP_{\Bf}$ has a unique maximal point $\rho=(\rho_1, \ldots, \rho_t)$ with $\rho_i \in \RR^{l_i}$ and set for $i\in\{1,\ldots,t\}$,
		$$S_i=\sup \left\{\ell\in\mathbb{Z}_{\geq 0} \mid \sum_{j=1}^{l_i} \rho_{i,j}^{(e)}  \leq p-1\quad\text{for every}\quad 0 \leq e \leq \ell  \right\}.$$ 
		If $I$ denotes the set of indices for which $S_i$ is finite, then the following assertions hold.
		\begin{enumerate}
			\item If $I= \emptyset$ then $\fpt(\fa)=\fpt(\fa^\circ)$. 
			\item If $I \ne \emptyset$ then 
			$$\fpt(\fa) \geq \sum_{i \in \{1,\ldots,t\} \setminus I}|\rho_i| + \sum_{i \in I}\left(|\langle \rho_i \rangle_{S_i}|+\frac{1}{p^{S_i}}\right).$$
		\end{enumerate}
\end{theoremx} 

On our second main result we contributed several cases in favor of a famous conjecture of Musta\c{t}\u{a}, Tagaki, and Watanabe \cite[Conjecture 3.6]{MTW} asserting  that there are infinitely many primes $p$ for which the $F$-pure threshold of the reduction mod $p$ of $\fa=(f_1,\ldots,f_t)$, $\fa_p$, equals the log canonical threshold of $\fa$. 
\begin{theoremx}[{Theorem \ref{Thm:fpt-equals-lct}}]
	Let $\fa \subseteq \m$ be a nonzero ideal with term ideal $\fa^\circ$, and take a list $\Bf=(f_1,\ldots,f_t)$  of minimal generators of $\fa$. Suppose that $\SP_{\Bf}$ has a unique maximal point $\rho=(\rho_1, \ldots, \rho_t)$ with $\rho_i \in \RR^{l_i}$. 
	\begin{enumerate}
		\item Assume that $\rho$ verifies the following property, whenever there is one $k_0\in\{1,\ldots,t\}$ with $|\rho_{k_0}|>1$, then $|\rho_k|>1$ for every $k\in\{1,\ldots,t\}$. If $\lct(\fa^\circ)>t$, then $\fpt(\fa_p)=\lct(\fa)=t$ for an infinite set of primes. 
		\item Assume that $\rho$ verifies the following property, whenever there is one $k_0\in\{1,\ldots,t\}$ with $|\rho_{k_0}|\leq 1$, then $|\rho_k|\leq 1$ for every $k\in\{1,\ldots,t\}$. Assume moreover that the set of prime numbers for which $\rho_k$ adds without carrying for every $k\in\{1,\ldots,t\}$ is an infinite set. If $\lct(\fa^\circ)\leq t$, then $\fpt(\fa_p)=\lct(\fa)$ for an infinite set of primes.
	\end{enumerate}
\end{theoremx}

Our third main result deals with a recent invariant of singularities in characteristic $p$ measuring the volume of the constancy regions for  generalized mixed test ideals \cite{BCNBRV}. This invariant is called the $F$-volume and again explicit computations are very hard. Using our construction of splitting polytopes for ideals  $\fa=(f_1,\ldots,f_t)$, we provide geometrical lower bounds for the $F$-volume of a collection of principal ideals.

\begin{theoremx}[{Theorem \ref{Thm:volume-polytope}}]	
		Consider a polynomial mapping $\boldsymbol{f}=(f_1,\ldots, f_t)$ of elements in $\m$. Suppose that $\SP_{\Bf}$ has a unique maximal point $\rho=(\rho_1, \ldots, \rho_t)$ with $\rho_i \in \RR^{l_i}$, and set for $i\in\{1,\ldots,t\}$,
	$$S_i=\sup \left\{\ell\in\mathbb{Z}_{\geq 0} \mid \sum_{j=1}^{l_i} \rho_{i,j}^{(e)}  \leq p-1\quad\text{for every}\quad 0 \leq e \leq \ell  \right\}.$$ 
	If $I$ denotes the set of indices for which $S_i$ is finite,  then
	$$ \prod_{i\in I} \left(|\langle \rho_i \rangle_{S_i}|+\frac{1}{p^{S_i}} \right)  \prod_{i\notin I} |\rho_i|\leq \Vol_F^{\m}((f_1),\ldots,(f_t)).$$ 
\end{theoremx}


\section{Splitting polytopes and Newton polyhedra of polynomial mappings}\label{Sec:Sec2}

In this section we associate two polytopes to a given ideal $\fa$ in a polynomial ring $R=\mathbb{k}[x_1, \ldots, x_m]$, with $\mathbb{k}$ an arbitrary field. One object is the usual Newton polytope of $\fa$, the other is the splitting polytope, see Definition~\ref{Def:ExpMat_SpliPol} for the details. The key result of this section is Proposition~\ref{propoBinomio} which allow us to control the coefficients of relevant monomials in precise powers of polynomials, by means of the splitting polytope. The terminology of $p$-expansions, that we describe next, will be useful in Section~\ref{Sec:Sec3}.

\subsection{Basics on $p$-expansions}
The following construction was introduced by R\'{e}nyi in \cite{Renyi} under the more general definition of $\beta$-expansions. It coincides essentially with the notion of nonterminating base $p$-expansion of \cite[Section 2]{HDpolynomials}.
\begin{definition}
	Let $p$ be an integer with $p\geq2$. Any real number $\alpha\in (0,1]$  can be uniquely written as
		\begin{equation}\label{Eq:p_expansion}
			\alpha=\mathlarger{\sum}_{k\geq 1}\frac{\alpha^{(k)}}{p^{k}},
		\end{equation}
		with $0\leq \alpha^{(k)} \leq p-1$, and the further assumption that the sequence $\{\alpha^{(k)}\}_{k\geq 1}$ does not eventually take only the value 0. We will say that the sequence $\{\alpha^{(k)}\}_{k\geq 1}$ represents $\alpha$ and call \eqref{Eq:p_expansion} the $p$-expansion of $\alpha$\footnote{Note that the $p$-expansion of $\alpha$ is different from its $p$-adic expansion, which is not defined for general real numbers.}.
\end{definition}
The representation $\{\alpha^{(k)}\}_{k\geq 1}$  is done through an infinite sequence which is obtained by using the `greedy' algorithm of R\'{e}nyi when $\alpha\neq \frac{a}{p^e}$, for an integer $a$. When $\alpha=\frac{a}{p^e}$ we choose the representation with $p-1$ in all the places after $e$. We denote by $\langle\alpha\rangle_e:=\sum_{k= 1}^e\frac{\alpha^{(k)}}{p^{k}},$ the $e$-th \textit{truncation} of $\alpha$ in \eqref{Eq:p_expansion}, and additionally we set $\alpha^{(0)}=\langle\alpha\rangle_0=\langle 0\rangle_e=0$ and $\langle\alpha\rangle_\infty=\alpha$.
\begin{definition}
	For $\alpha=(\alpha_1,\ldots,\alpha_m)\in\RR^m$ we define $|\alpha|$ as the sum $\alpha_1+\cdots+\alpha_m$ of the entries of $\alpha$. When $\alpha\in[0,1]^m$ we denote by $\langle\alpha\rangle_e$ the vector $(\langle\alpha_1\rangle_e,\ldots,\langle\alpha_m\rangle_e)$. We also say that $\alpha_1,\ldots,\alpha_m$ add without
	carrying modulo $p$ if $\alpha_1^{(k)}+\cdots+\alpha_m^{(k)}\leq p-1$ for every $k\geq 1$.
\end{definition}
If moreover $p$ is assumed to be a prime number, then the classical Lucas's Lemma~\cite{Lucas} asserts that $\binom{|r|}{r}\not \equiv 0 \bmod p$ if and only if the expansions in base $p$ of the entries of $r$ add without carrying.
\begin{corollary}\label{Cor:adding wo carrying}
	The entries of a vector $(\alpha_1,\ldots,\alpha_m)\in[0,1]^m$ add without carrying mod $p$ if and only if, for every $e\geq 1$
	\[\binom{p^e |\langle\alpha\rangle|}{p^e \langle\alpha\rangle_e}  \not \equiv 0 \bmod p.
	\]
\end{corollary}
\begin{lemma}[{\cite[Lemma 8]{HDpolynomials}}]\label{lema:lemma8}
	Consider a vector $(\alpha_1,\ldots,\alpha_m)$ in $\QQ^m\cap [0,1]^m$.
	\begin{enumerate}
		\item If $|\alpha|\leq 1$, then $\alpha_1,\ldots,\alpha_m$ add without carrying modulo infinitely many primes.
		\item If $|\alpha|> 1$, then $\alpha_1^{(1)}+\cdots+\alpha_m^{(1)} \geq p$, for $p$ large enough.
	\end{enumerate}
\end{lemma}
\begin{example}\label{Ex:pExpansion}
	If $\alpha\in (0,1]$ and $p$ is a prime number with the property $(p-1)\alpha\in \ZZ_{> 0}$, then $\alpha^{(e)}=(p - 1)\alpha$ for every $e\geq 1$ by~\cite[Lemma 4]{HDpolynomials}. Now fix a set of integers $A,B,C$ and consider a prime $p$ such that $p \equiv 1 \mod ABC$. Note that whenever $AB+BC+AC\leq ABC$, then $1/A$, $1/B$, and $1/C$ add without carrying modulo $p$.
\end{example}

\subsection{Splitting polytopes of polynomial mappings}\label{Ssec:Spli Pol}
Introduced first by Musta{\c{t}}{\u{a}}, Takagi and Watanabe \cite{MTW} for the case of one polynomial, the splitting polytope has been studied further by Hern\'andez \cite{HDbinomial,HDpolynomials}. Our construction for polynomial mappings follows closely the one given by Hern\'andez.

Given $\alpha, \beta \in \RR^m$, we denote by $\alpha \preceq \beta$ the component-wise inequality, and denote by $\boldsymbol{1}_m$ the vector $(1,\ldots,1)\in\RR^m$. Let $f=\sum_{a\in{\mathbb{Z}_{\geq 0}^m}} c_ax^a$ be a non-constant polynomial in $R$. The support of $f$ is the set 
\[\Supp(f)=\{a\in{\mathbb{Z}_{\geq 0}^m} \mid c_a\neq 0\}.\]
If $\Card(\Supp(f))=n$, then the $m\times n$ matrix $\E_f$ having as columns the elements of $\Supp(f)$ is called the \textit{exponent matrix} of $f$.
\begin{definition}\label{def:SpliPolf}
	The set $\SP_f:=\{\gamma \in \RR_{\geq 0}^{n} \mid \E_f \cdot \gamma\preceq \textbf{1}_{m}\}$, is called the \textit{splitting polytope of} $f$.
\end{definition} 
Consider now a finite list $\{f_1, \ldots, f_t\}$ of polynomials in $R$ belonging to $\m$. In order to define a splitting polytope for the polynomial mapping $\boldsymbol{f}=(f_1,\ldots, f_t)\, :\mathbb{k}^m\,\mapsto \mathbb{k}^t$ we need some technical assumptions. For any $1\leq i\leq t$, assume that $f_i$ has exactly $n_i$ monomials, i.e. $\Card(\Supp(f_i))=n_i$, and  
\begin{equation}\label{Eq:monomials_in_f}
	f_i=c_{a_{1}(i)}x^{a_{1}(i)}+\cdots+c_{a_{n_i}(i)}x^{a_{n_i}(i)}.
\end{equation}

Now we set $\hat{f}_1=f_1$ and define  $\hat{f}_2$, as the polynomial obtained from $f_2$ by removing the terms with monomials already appearing in $f_1$ regardless of the corresponding coefficients. We proceed with this construction and define inductively $\hat{f}_i$, as the polynomial obtained from $f_i$ by removing the terms with monomials already appearing in any of the polynomials $f_1, \ldots, f_{i-1}$, regardless of the corresponding coefficients. Note that the list of monomials of all the polynomials in $\{\hat{f}_1, \ldots, \hat{f}_t\}$ is reduced in the sense that it contains only the monomials in the elements of the list $\{f_1, \ldots, f_t\}$ without repetitions. We denote the exponent matrix of each element in the reduced list by $\tE_i$ and set $l_i=\Card(\Supp(\hat{f}_i))$. When $l_i=0$ we just consider the reduction of the elements in the list $\{f_1,\ldots,f_{i-1},f_{i+1},\ldots, f_t\}$.
\begin{definition}\label{Def:ExpMat_SpliPol}
Given a polynomial mapping $\Bf=(f_1, \ldots, f_t)$ we set $N=l_1+\cdots+l_t$ and define the following.
\begin{enumerate}
\item An \textit{exponent matrix} of $\Bf$ is an $(m \times N)$-matrix, lets say $\E_{\Bf}$, having as columns the different elements of $\bigcup_{i=1}^t\Supp(\hat{f}_i)$, i.e., $\E_{\Bf}=(\tE_{1}\mid \cdots\mid \tE_{t})$.
\item The \textit{splitting polytope} of $\Bf$ is the set 
\[\SP_{\Bf}=\{\gamma \in \RR_{\geq 0}^{N} \mid \E_{\Bf}\cdot \gamma \preceq \textbf{1}_{m}\}.\]
\end{enumerate}
\end{definition}
\begin{example}\label{Ex:List}
	For the polynomial mapping $\Bf=(x^{a}+y^b+z^c,-x^a+xyz+x^2y^2z^2,y^b+xyz+x^3y^2)$, an exponent matrix is 
	\[\E_{\Bf}=
	\begin{pmatrix}
		a & 0 & 0 &1 & 2 & 3\\
		0 & b & 0 &1 & 2 & 2\\
		0 & 0 & c &1 & 2 & 0
	\end{pmatrix}. 
	\]
\end{example}
Note that for a given polynomial mapping $\Bf=(f_1, \ldots, f_t)$ the exponent matrix is unique up to permutations. Also, the sets $\SP_f$ and $\SP_{\Bf}$ are bounded convex polytopes, in fact, they lie inside $[0,1]^n$ and $[0,1]^N$ respectively. Moreover since, for instance,  $\SP_f$ is a compact set, the function $|\cdot|:\, \RR^n\to\ \RR$ defined by $|(\gamma_1,\ldots,\gamma_n)|=\gamma_1+\cdots+\gamma_n$ reaches a maximum, let us say $M$. The set $\{\gamma\in \SP_f \mid |\gamma|=M\}$ defines a proper face of $\SP_f$, that is called maximal.

\begin{figure}[ht]
	\tdplotsetmaincoords{70}{115}
	\begin{tikzpicture}[scale=0.5,tdplot_main_coords]
		\draw[thick,->] (0,0,0)--(9,0,0) node[anchor=north east]{$\gamma_1$};
		\draw[thick,->] (0,0,0)--(0,7.5,0) node[anchor=north west]{$\gamma_2$};
		\draw[thick,->] (0,0,0)--(0,0,7) node[anchor=west]{$\gamma_3$};
		
		\draw[semithick,black,fill=cyan!30!blue!50] (6,0,4) -- (0,0,4)-- (0,2.25,4) -- (4.3125,2.25,4) -- cycle;
		\draw[semithick,black,fill=cyan!30!blue!50] (0,2.25,4) -- (4.3125,2.25,4) --(2.625,4.5,0) -- (0,4.5,0) -- cycle;
		\draw[semithick,black,fill=cyan!30!blue!50]
		(6,0,4) -- (6,0,0) --(2.625,4.5,0) -- (4.3125,2.25,4) -- cycle;
		\node[left] at (6,0,0.11){$v_1$};
		\node[left] at (6,0,4){$v_2$};
		\node[left] at (0,0,4.11){$v_3$};
		\node[above] at (0,2.25,4){$v_4$};
		\node[right] at (4.3125,2.36,3.9){$v_5$};
		\node[above] at (0,4.7,0){$v_6$};
		\node[below] at (2.625,4.5,0){$v_7$};
		\draw[dashed,black] (0,0,0)--(0,0,4) node[below left]{};
		\draw[dashed,black] (0,0,0)--(0,4.5,0) node[below left]{};
		\draw[dashed,black] (0,0,0)-- (6,0,0) node[below left]{};
	\end{tikzpicture}
	\caption{Splitting polytope of $\Bf=(x^{a}+xy^b,yz^c)$.}
	\label{Fig:ExSplitPol}
\end{figure}
		
\begin{example}\label{Ex:SplitPol}
	Consider the mapping $\Bf=(x^{a}+xy^b,yz^c)$ with $a,b,c \geq 1$, then 
	\begin{equation*}
	 \E_{\Bf}=
	\begin{pmatrix}
		a & 1 & 0\\
		0 & b & 1\\
		0 & 0 & c
	\end{pmatrix},    
	\end{equation*}
	and the splitting polytope of $\Bf$ is given by $\SP_{\Bf}=\{(\gamma_1,\gamma_2,\gamma_3)\in \RR_{\geq 0}^3 \mid a\gamma_1+\gamma_2 \leq 1, b\gamma_2+\gamma_3 \leq 1, \text{ and } c\gamma_3 \leq 1 \}$, see Figure \ref{Fig:ExSplitPol}. The polytope $\SP_{\Bf}$ is given by the convex hull of the vectors: $(0,0,0),\; v_1=\left(\frac{1}{a},0,0\right),\; v_2=\left(\frac{1}{a},0,\frac{1}{c}\right),\; v_3=\left(0,0,\frac{1}{c}\right),\; v_4=\left(0,\frac{c-1}{bc},\frac{1}{c}\right),\; v_5=\left(\frac{bc-c+1}{abc},\frac{c-1}{bc},\frac{1}{c}\right),\; v_6=\left(0,\frac{1}{b},0\right),$ and $v_7=\left(\frac{b-1}{ab},\frac{1}{b},0 \right)$. The maximal face of $\SP_{\Bf}$ depends on the values of $a$ and $c$, as follows. 
	\begin{itemize}
		\item If $c=1$ or if $c\neq1$ and $a \neq 1$, then the maximal face is $\{v_5\}$.
		\item  If $c\neq1$ and $a=1$, then the maximal face is the edge joining $v_2$ and $v_5$. 
	\end{itemize}
\end{example}

\begin{remark}
	Given a list of polynomials $\Bf$, one may be inclined to think that $\E_{\Bf}=\E_F$, where $F$ is the polynomial obtained by adding the different terms of each $f_i$ in $\Bf$. This is not the case, due to cancellations, cf. Example \ref{Ex:List}.
\end{remark}
For the following technical lemma we fix an integer $e\geq 1$, and recall that the $e$-th truncation of $\rho$ in its $p$-expansion is denoted by $\langle \rho \rangle_e$.

\begin{lemma}\label{equal-vector}
Assume that $\SP_{\Bf}$ has a unique maximal point $\rho\in [0,1]^N$, which is written as  $\rho=(\rho_1, \ldots, \rho_t)$ with $\rho_i \in \RR^{l_i}$, according to Definition \ref{Def:ExpMat_SpliPol}. Then the following statements hold.
\begin{enumerate}
\item If for some $\gamma \in \RR_{\geq 0}^N$, $|\gamma|=|\langle \rho \rangle_e|$ and $\E_{\Bf} \gamma = \E_{\Bf}  \langle \rho \rangle_e$,  then $ \gamma = \langle \rho \rangle_e $. Moreover, if 
there exist vectors $\phi, \psi \in \RR_{\geq 0}^{N}$, with $\psi \preceq \langle \rho \rangle_e$ and such that $|\phi|=|\psi|$ and  $\E_{\Bf} \phi = \E_{\Bf} \psi$;  then $\phi=\psi$.
\item Let $i$ be a fixed index in $\{1,\ldots,t\}$. If for some $\delta \in \RR_{\geq 0}^{l_i}$, $|\delta|=|\langle \rho_i \rangle_e|$ and $\tE_i \delta = \tE_i  \langle \rho_i \rangle_e$, then $ \delta = \langle \rho_i \rangle_e $. Moreover, if  there exist vectors $\phi, \psi  \in \RR_{\geq 0}^{l_i}$, with $\psi \preceq \langle \rho_i \rangle_e$ and such that $|\phi|=|\psi|$ and  $\tE_i \phi = \tE_i \psi$;  then $\phi=\psi$.
\end{enumerate}
\end{lemma}
\begin{proof}
 The case of just one polynomial is stated and proved in  \cite[Lemma 24]{HDpolynomials}. The proof given there can be easily extended to the present setting to justify the first part. Note, nevertheless, that a different argument is required to prove the second part, since our hypotheses do not imply in general that $\rho_i$ is a maximal point of $\SP_{\hat{f}_i}$ (see Remark \ref{rmk:P_LvsP_i}). 
 
 So we  start by considering the vector
 \[\gamma=(\langle \rho_1 \rangle_e, \ldots, \langle \rho_{i-1}\rangle_e,\delta,\langle \rho_{i+1} \rangle_e,\ldots,\langle \rho_{t} \rangle_e ) \in \RR_{\geq 0}^N.
 \]
It is clear that $|\gamma|=|\langle \rho \rangle_e|$, since $|\delta|=|\langle \rho_i \rangle_e|$. From $\tE_i \delta = \tE_i  \langle \rho_i \rangle_e$ follows
 \[\E_{\Bf} \gamma=\tE_1 \langle \rho_1 \rangle_e + \cdots + \tE_{i-1}\langle \rho_{i-1}\rangle_e + \tE_i \delta + \tE_{i+1} \langle \rho_{i+1} \rangle_e + \cdots + \tE_t \langle \rho_{t} \rangle_e =\E_{\Bf} \langle \rho \rangle_e,
 \]
and the first part of the lemma gives $\gamma= \langle \rho \rangle_e$, hence $\delta=\langle \rho_i \rangle_e$. For the second claim assume that there exist  $\phi, \psi \in \RR_{\geq 0}^{l_i}$, with $\psi \preceq \langle \rho_i \rangle_e$ and such that $|\phi|=|\psi|$ and $\tE_i \phi = \tE_i \psi$. Taking  $\phi^\prime=\phi+\langle \rho_i \rangle_e-\psi$ gives  $|\phi^\prime|= |\langle \rho_i \rangle_e|$ and $\tE_i \phi^\prime= \tE_i \langle \rho_i \rangle_e$. The first claim of part (2) implies $\phi^\prime=\langle \rho_i \rangle_e$, and consequently $\phi=\psi$.
\end{proof}

The following proposition provides an explicit formula for the coefficient of a relevant monomial in a polynomial $f$ raised to the power $p^{e}|\langle \rho_i \rangle_e|$. Here we assume that in $\Bf=(f_1, \ldots, f_t)$, every member can be written as in \eqref{Eq:monomials_in_f}, i.e. $\Card(\Supp(f_i))=n_i$.

\begin{proposition} \label{propoBinomio}
	 Under the assumptions of Lemma \ref{equal-vector}, the following statements hold.
\begin{enumerate}
\item The coefficient of the monomial $x^{p^{e}\tE_i \langle \rho_i \rangle_e}$ in $f_i^{p^{e} |\langle \rho_i \rangle_e|}$ is $\binom{p^{e}|\langle \rho_i \rangle_e|}{p^{e}\langle \rho_i \rangle_e}c_{a(i)}^{b_{e}(i)}$, where $b_{e}(i)$ is the vector in $\ZZ_{\geq 0}^{n_i}$ having $p^{e}\langle \rho_i \rangle_e$ in the first $l_i$ entries and $0$ in the remaining ones. Moreover, the coefficient of the monomial $x^{p^{e}\E_{\Bf} \langle \rho \rangle_e}$ in the polynomial $g:=\prod_{i=1}^{t} f_i^{p^{e}|\langle \rho_i \rangle_e|}$ is $\prod_{i=1}^{t}\binom{p^{e}|\langle \rho_i \rangle_e|}{p^{e}\langle \rho_i \rangle_e}c_{a(i)}^{b_{e}(i)}$.
\item Let $v$ be an element of  $\frac{1}{p^{e}}\ZZ_{\geq 0}^N$, satisfying $v \preceq \langle \rho \rangle_e$ and written in the form $v=(v_1, \ldots, v_t)$ with $v_i \in \frac{1}{p^{e}}\ZZ_{\geq 0}^{l_i}$. Then, the coefficient of the monomial $x^{p^{e}\tE_i v_i}$ in $f_i^{p^{e}|v_i|}$ is $\binom{p^{e}|v_i|}{p^{e} v_i}c_{a(i)}^{d_{e}(i)}$, where $d_{e}(i)$ is a vector of the form $p^e(v_i,0,\ldots,0)\in\ZZ_{\geq 0}^{n_i}$. Moreover, the coefficient of the monomial $x^{p^{e}\E_{\Bf} v}$ in the polynomial $g=\prod_{i=1}^{t} f_i^{p^{e}|v_i|}$ is $\prod_{i=1}^{t}\binom{p^{e}|v_i|}{p^{e}v_i }c_{a(i)}^{d_{e}(i)}$.
\end{enumerate}
\end{proposition}

\begin{proof}
	Consider a polynomial $h$ with $\Card(\Supp(h))=n$ and write $h=c_{a_{1}}x^{a_{1}}+\cdots+c_{a_{n}}x^{a_{n}}$. The multinomial theorem implies that for a positive integer $Q$,
	\begin{equation}\label{eq:multinomial}
		h^Q=(c_{a_{1}}x^{a_{1}}+\cdots+c_{a_{n}}x^{a_{n}})^Q=\sum_{|k|=Q} \binom{Q}{k}c_{a}^k x^{\E k},
	\end{equation}
	where we use the convention $\E:=\E_h$. We denote the set of indices in the sum \eqref{eq:multinomial} by $I=\{k\in \ZZ_{\geq 0}^{n}\ \text{;}\ |k|=Q\}$. Let $\widehat{I}=\widehat{I}(l)$ denote the subset of $I$ composed by vectors of the form $(k_1,\ldots,k_{l},0,\ldots,0)$, for some $l\leq n$. Moreover, denote by $\hat{k}$ the projection of $k \in \ZZ_{\geq 0}^{n}$ onto its first $l$ entries. Then \eqref{eq:multinomial} becomes:
	\begin{equation*}
			h^Q=\sum_{k\in \widehat{I}} \binom{Q}{k}c_{a}^k x^{\E k}+\sum_{k\in I\setminus \widehat{I}} \binom{Q}{k}c_{a}^k x^{\E k}
			=\sum_{k\in \widehat{I}} \binom{Q}{\hat{k}}c_{a}^{\hat{k}} x^{\tE \hat{k}}+\sum_{k\in I\setminus \widehat{I}} \binom{Q}{k}c_{a}^k x^{\E k},
	\end{equation*}
	where $\tE$ is the matrix obtained from $\E$ by deleting the last $n-l$ columns. We will further split the first sum over $\widehat{I}$ by fixing a vector $v\in \RR_{\geq 0}^m$, namely
	\begin{equation*}
		h^Q
		=\sum_{\substack{k\in \widehat{I}\\\tE \hat{k}=v}} \binom{Q}{\hat{k}}c_{a}^{\hat{k}} x^{\tE \hat{k}}+\sum_{\substack{k\in \widehat{I}\\\tE \hat{k}\neq v}} \binom{Q}{\hat{k}}c_{a}^{\hat{k}} x^{\tE \hat{k}}+\sum_{k\in I\setminus \widehat{I}} \binom{Q}{k}c_{a}^k x^{\E k}.
	\end{equation*}
	In order to show the first part of the Proposition, we fix $i\in\{1,\ldots,t\}$ and use the previous formulation with $h=f_i$, $Q=p^e|\langle \rho_i \rangle_e|$ and $v=p^{e}\tE_i \langle \rho_i \rangle_e$, obtaining   
	\begin{align*}		
		f_i^{p^{e}|\langle \rho_i \rangle_e|}
		&=\binom{p^{e}|\langle \rho_i \rangle_e|}{p^{e}\langle \rho_{i,1} \rangle_e,\ldots,p^{e}\langle \rho_{i,l_i} \rangle_e}
		c_{a}^{\hat{k}}x^{p^{e}\tE_i \langle \rho_i \rangle_e}+\sum_{\substack{k\in \widehat{I}\\\tE_{i}\hat{k}\neq p^e\tE_{i}\langle \rho_i \rangle_e}}  \binom{p^{e}|\langle \rho_i \rangle_e|}{\hat{k}}c_{a}^{\hat{k}} x^{\tE_{i} \hat{k}}\\
		&+\sum_{k\in I\setminus \widehat{I}} \binom{p^{e}|\langle \rho_i \rangle_e|}{k}c_{a}^k x^{\E_{f_i} k}.
	\end{align*}
	The first term in the right hand side contains just one summand due to Lemma \ref{equal-vector} (2) with $\delta = \frac{1}{p^{e}}\hat{k}$. The first part will be proved once we show that the sum over $I\setminus \widehat{I}$ does not contribute to the monomial $x^{p^{e}\tE_i \langle \rho_i \rangle_e}$. Assume otherwise that there exists $k \in \RR^{n_i}$ verifying $p^{e}k \in I\setminus \widehat{I}$ and $\E_{f_i}k=\tE_i \langle \rho_i \rangle_e $. Suppose without loss of generality that $k=(k_1,\ldots,k_{l_i},k_{l_{i+1}},0,\ldots,0)$ with $k_{l_{i+1}}\neq 0$. Hence, $\E_{f_i}k=\tE_i \hat{k}+\E'k_{l_{i+1}}$, where $\E'$ is a column matrix. In fact, due to the construction of $\E_{\Bf}$, $\E'$ is a column in the exponent matrix of some $\{f_1,\ldots,f_{i-1}\}$, so we will safely assume that $\E'$ is the first column of $\E_{f_1}=\tE_1$. Consider now the following vector
	\[\gamma=(\langle \rho_1 \rangle_e+(k_{l_{i+1}},\overbrace{0,\ldots,0}^{l_1-1}),\langle \rho_{2} \rangle_e, \ldots, \langle \rho_{i-1}\rangle_e,\hat{k},\langle \rho_{i+1} \rangle_e,\ldots,\langle \rho_{t} \rangle_e ) \in \RR_{\geq 0}^N.
	\]
	Clearly $|\gamma|=|\langle \rho \rangle_e|$, and since $\tE_i \langle \rho_i \rangle_e =\tE_i \hat{k}+\E'k_{l_{i+1}}$, we get
	\begin{equation*}
		\E_{\Bf} \gamma =\sum_{\substack{j=1\\j\neq i}}^t \tE_j \langle \rho_j \rangle_e + \tE_i \hat{k}+\E'k_{l_{i+1}}=\E_{\Bf} \langle \rho \rangle_e.  
	\end{equation*}
	But then the first part of Lemma \ref{equal-vector} implies $\gamma= \langle \rho \rangle_e$, in particular $|\hat{k}|=|\langle \rho_i \rangle_e|$, which contradicts the fact that $|k|=|\langle \rho_i \rangle_e|$, since $k_{l_{i+1}}\neq 0$. 
	
	The proof for the coefficient of $g:=\prod_{i=1}^{t} f_i^{p^{e}|\langle \rho_i \rangle_e|}$ now follows from a straightforward calculation and again the first part of Lemma \ref{equal-vector}. Finally, similar considerations lead to the proof of the second part. 
\end{proof}

\begin{remark}\label{rmk:P_LvsP_i}
	Let $\SP_{\Bf}$ be the splitting polytope of $\Bf=(f_1, \ldots, f_t)$. Assume that the unique maximal point $\rho$ of $\SP_{\Bf}$  is written as $\rho=(\rho_1, \ldots, \rho_t)$ with $\rho_i \in \RR^{l_i}$. Definitions \ref{def:SpliPolf} and \ref{Def:ExpMat_SpliPol} yield $\rho_i\in\SP_{\widehat{f_i}}$, but in general, the `maximality' of a point is not preserved by `projection'. Consider for example the mapping $\Bf=(f,g)$ given by $f=x^{a}+xy^b$ and $g=x^{a}+yz^c$. Assume further that $a,c \geq 2$ and $b\geq 1$. In this case $l_1=2,\ l_2=1, N=3$, and  $\SP_{\Bf}$ coincides with the splitting polytope of Example \ref{Ex:SplitPol}. One has 
	$\E_{f}=
	\begin{pmatrix}
		a & 1 \\
		0 & b 
	\end{pmatrix}$ and $\SP_{f}=\{(\gamma_1,\gamma_2)\in \RR_{\geq 0}^2 \mid a\gamma_1+\gamma_2 \leq 1, \text{ and } b\gamma_2\leq 1\}$, see Figure \ref{Fig:ExSplitPol of f_1}. 
The maximal point of $\SP_{\Bf}$ is  $\left(\frac{bc-c+1}{abc},\frac{c-1}{bc},\frac{1}{c}\right)$ and its  `projection' on the first two coordinates is $\left(\frac{bc-c+1}{abc},\frac{c-1}{bc}\right)$. The latter point belongs to the boundary of $\SP_{f}$, but in general is not equal to its maximal point $\left(\frac{b-1}{ab},\frac{1}{b}\right)$. In Figure \ref{Fig:ExSplitPol of f_1}, the point $\left(\frac{bc-c+1}{abc},\frac{c-1}{bc}\right)$ is depicted in red color.

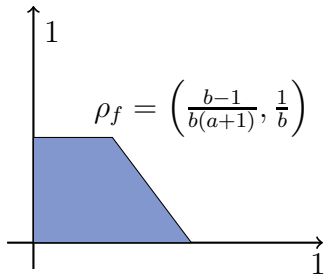
\begin{figure}[hb]
		\begin{tikzpicture}[scale=0.7]
			\draw[thick,black,->] (-0.5,0)--(5.5,0) node[right,below] {}; 
			\foreach \x/\xtext in {} 
			\draw[shift={(\x,0)},blue] (0pt,2pt)--(0pt,-2pt) node[below] {$\xtext$};
			%
			\foreach \y/\ytext in {}
			\draw[shift={(0,\y)},blue] (2pt,0pt)--(-2pt,0pt) node[left] {$\ytext$};
			\draw[thick,black,->] (0,-0.5)--(0,4.5) node[left,above] {}; 
			%
			\draw[black,fill=cyan!30!blue!50] (0,0) -- (3,0) -- (1.5,2)--(0,2)--(0,0);
			%
			\draw[fill=black] (3/2,2) circle[radius=2pt];
			\draw[fill=red] (2,4/3) circle[radius=2pt];
			\node[below] at (5.4,0) {$1$};
			
			\node[right] at (0,4) {$1$};
			\node[above] at (3.2,1.8) {$\rho_f=\left(\frac{b-1}{ab},\frac{1}{b}\right)$};
			\end{tikzpicture}
		\caption{Splitting polytope of $f=x^{a}+xy^b$.}
		\label{Fig:ExSplitPol of f_1}
	\end{figure}
\end{remark}

\subsection{Newton polyhedra of polynomial mappings}\label{SSec:Newton}
Our definition of $\SP_{\Bf}$ was inspired by the construction of the Newton polyhedron $\N_{\Bf}=\N_{(f_1, \ldots, f_t)}$  of a polynomial  mapping $\boldsymbol{f}=(f_1,\ldots, f_t)\, :\mathbb{k}^m\,\mapsto \mathbb{k}^t$ that we address next assuming the notations of Section \ref{Ssec:Spli Pol}.  

\begin{definition} \label{defiNewton}
	Let $A$ be a nonempty subset of $\mathbb{Z}_{\geq 0}^{m}$. The Newton polyhedron 
	$\N_A$ associated to $A$ is the convex hull in
	$\mathbb{R}_{\geq 0}^{m}$ of the set 
	$$\bigcup_{a\in A}\left(  a+\mathbb{R}_{\geq 0}^{m}\right).$$ 
\end{definition}
Classically one associates a Newton	polyhedron $\N_f$ to a polynomial $f=\sum_{a\in\mathbb{Z}_{\geq 0}^m} c_ax^a$ in the maximal ideal $\m$ of $R=\mathbb{k}[x_1, \ldots, x_m]$ by setting $A=\Supp(f)$. There are two customary notions of `genericity' for a fixed $\N_{f}$. We say that $f$ is \textit{non degenerated with respect to} $\N_f$ if for any compact face $\eta$ of $\N_f$ the system of equations
\begin{equation}\label{Eq:Nondeg}
	f_\eta=\partial f_\eta/\partial x_1=\cdots=\partial f_\eta/\partial x_m=0
\end{equation}
has no solution in $(\mathbb{k}^\times)^m$, where $f_\eta$ denotes the polynomial obtained from $f$ by restricting the support to $\eta$. For the second genericity condition, $\N_f$ is called \textit{convenient} when the support of $f$ contains a point in every coordinate axis. A third condition is given by Hernández \cite[Definition 29]{HDpolynomials} where he says that $\N_f$ is in \textit{diagonal position} if the ray spanned by the vector $\boldsymbol{1}_m$ intersects a compact face of $\N_f$. It is elementary to show, following Kouchnirenko's ideas \cite{Kou}, that this condition is also generic (in the sense of Kouchnirenko). Note moreover that every convenient $\N_f$ will be automatically in diagonal position. 

Consider now a polynomial  mapping  $\boldsymbol{f}=(f_1,\ldots, f_t)\, :\mathbb{k}^m\,\mapsto \mathbb{k}^t$. We denote by $\N_{\Bf}$ the Newton polyhedron of $\boldsymbol{f}=(f_1,\ldots, f_t)$, which is given by using Definition \ref{defiNewton} with  $A=\bigcup_{i=1}^t\Supp(f_i)=\bigcup_{i=1}^t\Supp(\hat{f}_i)$. Moreover we consider the ideal $\fa=\fa_{\boldsymbol{f}}=(f_1,\ldots, f_t)$ generated by the elements of the polynomial mapping $\Bf$ and define $\N_{\fa}$ as the Newton polyhedron of the set $A=\bigcup_{f\in\fa}\Supp(f)$. Obviously $\N_{\fa}=\N_{\Bf}$, which in particular shows that $\N_{\Bf}$ depends only on $\fa$ and not on a system of generators, see \cite[Section 2.1]{CNV15}. Discussions, comparisons and examples of these definitions (and the corresponding non degeneracy conditions that they imply) are given in \cite{BiFuSa,Tei,VeZu,LeVeZu,Ros,Ste}.

Some relations between $\N_{\fa}=\N_{\Bf}$ and $\SP_{\Bf}$ for principal ideals have been treated by Hernández in \cite[Section 4.4]{HDpolynomials}. It is not difficult to adapt those ideas for non-principal ideals as we discuss in the remainder of this Section.
\begin{proposition}\label{Prop:SplPolvsNewPol}
	The following equality holds
	$$\left\lbrace |\gamma|  \mid \gamma\in \SP_{\Bf} \setminus \{(0,\ldots,0)\} \right\rbrace=\left\{\tau \in \RR_{> 0}\mid (1/\tau,\ldots,1/\tau)\in \N_{\Bf} \right\}.$$ 
	In particular, $\max\left\lbrace |\gamma|  \mid \gamma\in \SP_{\Bf} \setminus \{(0,\ldots,0)\} \right\rbrace=\max\left\{\tau \in \RR_{> 0}\mid (1/\tau,\ldots,1/\tau)\in \N_{\Bf} \right\}.$
\end{proposition}
\begin{proof}
	For $\gamma \in \SP_{\Bf} \setminus \{(0,\ldots,0)\}$ we have $\E_{\Bf}\cdot \gamma  \preceq \textbf{1}_{m}$ and then we choose $\delta \in \RR_{\geq 0}^m$ such that $\E_{\Bf} \gamma + \delta = \textbf{1}_{m}$. By writing $\gamma=(\gamma_1,\ldots, \gamma_t) \in \RR^{l_1}\times\cdots\times\RR^{l_t}$, we get  
	\begin{equation*}
		\textbf{1}_{m}=\tE_1 \gamma_1 + \cdots + \tE_t \gamma_t + \delta=\sum_{j_1=1}^{l_1}\gamma_{1,j_1}a_{1,j_1}+ \cdots + \sum_{j_t=1}^{l_t}\gamma_{t,j_t}a_{t,j_t} + \delta.
	\end{equation*}
	Setting $\tau =|\gamma| >0$ yields  $\frac{1}{\tau} \textbf{1}_{m}=(1/\tau,\ldots,1/\tau)=\sum_{i=1}^{t}\sum_{j_i=1}^{l_i}\frac{\gamma_{i,j_i}}{\tau} a_{i,j_i}+\frac{1}{\tau}\delta$, but $a_{i,j_i} \in \bigcup_{i=1}^{t}\Supp(f_i)$ and $\sum_{i=1}^{t}\sum_{j_i=1}^{l_i}\frac{\gamma_{i,j_i}}{\tau}=\frac{|\gamma|}{\tau}=1$, which forces $(1/\tau,\ldots,1/\tau) \in \N_{\Bf}$. The other inclusion follows by similar considerations.
\end{proof}
Under the assumption that $\N_{\Bf}$  is in diagonal position, Proposition \ref{Prop:SplPolvsNewPol} offers a very useful tool for determining the maximal face of  the splitting polytope $\SP_{\Bf}$.
\begin{proposition}\label{Prop:maximal face}
Assume that $\N_{\Bf}$ is in diagonal position and let $\eta$ be the face of $\N_{\Bf}$ that is intersected by the ray spanned by $\textbf{1}_{m}$.  Then the maximal face of $\SP_{\Bf}$ is given by $$\{\gamma \in \RR_{\geq 0}^{N} \mid \E_{\Bf}\cdot \gamma = \textbf{1}_{m} \;\text{and}\; \gamma_{j}=0 \; \text{if the j-th column of } \; \E_{\Bf} \; \text{does not belongs to} \; \eta\}.$$ 
\end{proposition}
\begin{proof}
	The proof is an easy variation of the one given for the case of a principal ideal \cite[Lemma 32]{HDpolynomials}.
\end{proof}
\begin{example}\label{Ex:NewtPolToMaxFace}
	Let $\Bf$ be the polynomial mapping given by $(y^{5}(x+y^m)^4,x^n)$ with $m\geq 1$ and $n > 4+5/m$. Then 
	$\E_{\Bf}=
		\begin{pmatrix}
			4 & 3 & 2 & 1 & 0 & n\\
			5 & m+5 & 2m+5 & 3m+5 & 4m+5 & 0
		\end{pmatrix},    
	$
	and $\SP_{\Bf}=\{(\gamma_1,\gamma_2,\gamma_3,\gamma_4,\gamma_5,\gamma_6)\in \RR_{\geq 0}^6 \mid 4\gamma_1+3\gamma_2+2\gamma_3+\gamma_4+n\gamma_6\leq 1, \text{ and }
	5\gamma_1+(m+5)\gamma_2+(2m+5)\gamma_3+(3m+5)\gamma_4+(4m+5)\gamma_5\leq 1 \}$. The ray spanned by the vector $(1,1)$ intersects $\N_{\Bf}$ in the face formed by the line segment joining $(4,5)$ and $(n,0)$. Proposition \ref{Prop:maximal face}  gives that $\SP_{\Bf}$ has a unique maximal point with equations 
	$\{(\gamma_1,0,0,0,0,\gamma_6)\in \RR_{\geq 0}^6 \mid 4\gamma_1+n\gamma_6 = 1 \text{ and }	5\gamma_1= 1\}=\{(1/5,0,0,0,0,1/5n)\}.$ Note that $\left|(\frac{1}{5},0,0,0,0,\frac{1}{5n})\right|=\frac{1}{5}+\frac{1}{5n}$ which equals $\lct((y^{5}(x+y^m)^4,x^n))$, according to \cite[Theorem 6.5]{CNV15}.
\end{example}


\section{$F$-pure thresholds of non-principal ideals}\label{Sec:Sec3}

In this section we present two of the main results of this work Theorems~\ref{Thm:F-pure-threshold-polytope} and~\ref{Thm:fpt-equals-lct}. A relevant role is played here by the term ideal of a given ideal $\fa$ which is presented in Definition~\ref{Def:Term Ideal}. Then we discuss the relations between the log canonical threshold and the $F$-pure threshold of $\fa$. Several examples are presented to illustrate the scope of the discussed techniques.

\subsection{$F$-pure thresholds and term ideals }

Take a polynomial ring $R$ over a field $\mathbb{k}$ of prime characteristic $p$, and set $\m=(x_1, \ldots, x_m)$ for the homogeneous maximal ideal of $R$. For $e\geq 1$ we denote by $\m^{[p^{e}]}$ the $p^{e}$-th Frobenius power of $\m$, i.e. the ideal $(x_1^{p^e}, \ldots, x_m^{p^e})$.

\begin{definition}\label{Def:fpt}
	For a nonzero ideal $\fa \subseteq \m$, the following integer is well defined 
	\[\nu_{\mathfrak{a}}(p^e)= \max \{c \in \mathbb{Z}_{\geq 0} \mid \mathfrak{a}^c \not \subseteq \m^{[p^{e}]}\}.
	\]
	We also define the $F$-pure threshold of $\mathfrak{a}$ (at the origin) by
	\[\fpt(\mathfrak{a})= \lim\limits_{e \rightarrow \infty} {\frac{\nu_{\mathfrak{a}}(p^e)}{p^e}}.
	\]
\end{definition}
The existence of the previous limit is proved in \cite{MTW}, while the non obvious fact that $\fpt(\fa)$ is a rational number is proved in \cite{BMS-MMJ}. 
\begin{example}\label{Ex:Fthrsholds}
 It is not difficult to show that if $\fa$ defines a non singular variety of codimension $r$, then $\fpt(\fa) = r$. Also it is known \cite{How,HaYo} that for a monomial ideal $\fa$ its $F$-pure threshold is given by the formula 
 \[\fpt(\fa)=\max \{\tau \in \RR_{\geq 0}\mid \boldsymbol{1}_m \in \tau\cdot \N_{\fa}\}.\] 
 This is exactly the same number obtained by computing the log canonical threshold of the monomial ideal (in characteristic zero, lets say over $\QQ$) generated by the monomials in $\fa$. The equality is a consequence of the stronger fact that in the monomial case, the multiplier and test ideals of $\fa$ coincide.
 \end{example}


\begin{definition}[\cite{BiFuSa}]\label{Def:Term Ideal}
	For a nonzero ideal $\fa \subseteq \m$, the term ideal $\fa^\circ$ of $\fa$ is the ideal generated by all the monomials $x^a$ with $a\in \N_{\fa}$.
\end{definition}

Take a non necessarily principal ideal $\fa\subseteq \m$, and consider the polynomial mapping $\Bf=(f_1,\ldots,f_t)$ formed by a list of minimal generators of $\fa$. 
\begin{proposition}\label{Prop:Properties ftt ideals}The following properties hold.
	\begin{enumerate}
		\item $\fpt(\fa)\leq \min\{t, \fpt(\fa^\circ)\}$.
		\item $\fpt(\fa^\circ)=\max_{\gamma\in \SP_{\Bf} } |\gamma|.$
	\end{enumerate}
\end{proposition}
\begin{proof}
	The first property is well known, see e.g. \cite{MTW}. The second property follows from the discussion in Example \ref{Ex:Fthrsholds} and Proposition  \ref{Prop:SplPolvsNewPol}, alternatively one may adapt the proof for a principal ideal \cite[Proposition $36$]{HDpolynomials}.
\end{proof}
Our next goal is to make explicit the conditions required to get $\fpt(\fa)=\fpt(\fa^\circ)$. Our formulation extends previous results only available for principal ideals \cite[Theorem 42]{HDpolynomials}  and \cite[Theorem 4.1]{HDbinomial}. 
\begin{definition}\label{Def:S-i}
	Given a vector $\rho=(\rho_1, \ldots, \rho_t)$ with $\rho_i \in \RR^{l_i}$, we set for $i\in\{1,\ldots,t\}$
	$$S_i=\sup \left\{\ell\in\mathbb{Z}_{\geq 0} \mid \sum_{j=1}^{l_i} \rho_{i,j}^{(e)}  \leq p-1\quad\text{for every}\quad 0 \leq e \leq \ell  \right\}.$$ 
\end{definition}

\begin{theorem} \label{Thm:F-pure-threshold-polytope}
	Let $\fa \subseteq \m$ be a nonzero ideal with term ideal $\fa^\circ$, and take a list $\Bf=(f_1,\ldots,f_t)$  of minimal generators of $\fa$. Suppose that $\SP_{\Bf}$ has a unique maximal point $\rho=(\rho_1, \ldots, \rho_t)$ with $\rho_i \in \RR^{l_i}$.
	
	\noindent If $I$ denotes the set of indices for which $S_i$ is finite, then the following assertions hold.
	\begin{enumerate}
		\item If $I= \emptyset$ then $\fpt(\fa)=\fpt(\fa^\circ)$. 
		\item If $I \ne \emptyset$ then 
		\[
			\fpt(\fa) \geq \sum_{i \in \{1,\ldots,t\} \setminus I}|\rho_i| + \sum_{i \in I}\left(|\langle \rho_i \rangle_{S_i}|+\frac{1}{p^{S_i}}\right).
		\]
	\end{enumerate}
\end{theorem}
\begin{proof}
	For the case $I= \emptyset$, it will be enough to show that $|\rho| \leq \fpt(\fa)$, according to Proposition~\ref{Prop:Properties ftt ideals}. Our first goal is to find an exponent $Q(e)$ such that $\fa^{Q(e)} \nsubseteq  \m^{[p^{e}]}$. In order to check that this exponent is $p^{e} |\langle \rho \rangle_e|$ we start by recalling from Proposition~\ref{propoBinomio}(1) that
	\begin{equation}\label{Eq:Monomial in g}
		\prod_{i=1}^{t}\binom{p^{e}|\langle \rho_i \rangle_e|}{p^{e}\langle \rho_i \rangle_e}c_{a(i)}^{b_{e}(i)}x^{p^{e}\E_{\Bf} \langle \rho \rangle_e}
	\end{equation}
	appears as a summand of the polynomial  $g=\prod_{i=1}^{t} f_i^{p^{e}|\langle \rho_i \rangle_e|}$. Since $I=\emptyset$,  Corollary \ref{Cor:adding wo carrying} implies $\binom{p^{e}|\langle \rho_i \rangle_e|}{p^{e}\langle \rho_i \rangle_e} \not \equiv 0$ mod $p$ for each $\rho_i$, making the coefficient of \eqref{Eq:Monomial in g} not zero and thus $p^{e}\E_{\Bf} \langle \rho \rangle_e \in \Supp(g)$. However, from $\langle \rho \rangle_e \prec \rho$ we have 
	\[p^{e}\E_{\Bf} \langle \rho \rangle_e \prec p^{e}\E_{\Bf} \rho \preceq p^{e}\textbf{1}_m,
	\]
	meaning that  $x^{p^{e}\E_{\Bf} \langle \rho \rangle_e}$ does not belong to 
	 $\m^{[p^{e}]}$. In turn this implies that $g \notin \m^{[p^{e}]}$ and finally $\fa^{p^{e} |\langle \rho \rangle_e|} \nsubseteq  \m^{[p^{e}]}$ because $g$ was already a member of $\fa^{p^{e}|\langle \rho \rangle_e|}$. We conclude that $|\langle \rho \rangle_e| \leq \frac{\nu_{\fa}(p^{e})}{p^{e}}$, which take us to $|\rho| \leq \fpt(\fa)$ by taking the limit as $e \to \infty$.
	
	We now address the second part of the proof. We assume without lost of generality that $I=\{1,\ldots,r\}$ and define $S=\max_{1 \leq i \leq r}S_i$.  
	Our strategy here is to find again a suitable exponent $Q(e)$ verifying $\fa^{Q(e)} \nsubseteq  \m^{[p^{e}]}$, but in this case the construction of the corresponding monomial is more involved. Note first that the definition of $S_i$ gives
\begin{equation*}
\begin{matrix}
\rho_{1,1}^{(S_1+1)}+ \cdots + \rho_{1,l_1}^{(S_1+1)} \geq p\\
\rho_{2,1}^{(S_2+1)}+ \cdots + \rho_{2,l_2}^{(S_2+1)} \geq p\\
\vdots \\
\rho_{r,1}^{(S_r+1)}+ \cdots + \rho_{r,l_r}^{(S_r+1)} \geq p.
\end{matrix}
\end{equation*}
There should exist then non negative integers $n_{1,1}, \ldots, n_{1,l_1},$  $n_{2,1}, \ldots, n_{2,l_2},$ $\ldots, n_{r,1}, \ldots, n_{r,l_r}$ such that
\begin{equation*}
\begin{matrix}
n_{1,1}+ \cdots + n_{1,l_1} = p-1 &  \mathrm{and}& 0 \leq n_{1,j_1} \leq \rho_{1,j_1}^{(S_1+1)}\\
n_{2,1}+ \cdots + n_{2,l_2}= p-1 &  \mathrm{and}&0 \leq n_{2,j_2} \leq \rho_{2,j_2}^{(S_2+1)}\\
\vdots & \vdots & \vdots\\
n_{r,1}+ \cdots + n_{r,l_r} = p-1 &  \mathrm{and}& 0 \leq n_{r,j_r} \leq \rho_{r,j_r}^{(S_r+1)},
\end{matrix}
\end{equation*}
with at least one index in each row verifying the strict inequality, say $n_{i,1}<\rho_{i,1}^{(S_i+1)}$. 
For $e$ big enough we define the vectors, 
\[
\begin{cases}
		v_1=\langle \rho_1 \rangle_{S_1}+\left(\frac{n_{1,1}}{p^{S_1+1}}+\frac{p-1}{p^{S_1+2}}+ \cdots + \frac{p-1}{p^{S_1+e}},\frac{n_{1,2}}{p^{S_1+1}}, \ldots, \frac{n_{1,l_1}}{p^{S_1+1}} \right)\\
		v_2=\langle \rho_2 \rangle_{S_2}+\left(\frac{n_{2,1}}{p^{S_2+1}}+\frac{p-1}{p^{S_2+2}}+ \cdots + \frac{p-1}{p^{S_2+e}},\frac{n_{2,2}}{p^{S_2+1}}, \ldots, \frac{n_{2,l_2}}{p^{S_2+1}} \right)\\
         \vdots \\
		v_r=\langle \rho_r \rangle_{S_r}+\left(\frac{n_{r,1}}{p^{S_r+1}}+\frac{p-1}{p^{S_r+2}}+ \cdots + \frac{p-1}{p^{S_r+e}},\frac{n_{r,2}}{p^{S_r+1}}, \ldots, \frac{n_{r,l_r}}{p^{S_r+1}} \right).
\end{cases}
\] 
In this way $v =(v_1,\ldots,v_r) \in \RR_{\geq 0}^{l_1+\cdots+l_r}$ verifies $p^{S+e}v \in \ZZ_{\geq 0}^{l_1+\cdots+l_r}$ and moreover
\begin{align*}
|v|&=\sum_{i=1}^{r}\left(|\langle \rho_i \rangle_{S_i}|+\frac{p^{e}-1}{p^{S_i+e}} \right).
\end{align*}
Observe also that  for $i\in\{1,\ldots,r\}$, $v_i \preceq \langle \rho_i \rangle_{S_i+e} \preceq \langle \rho_i \rangle_{S+e}$ and by Proposition~\ref{propoBinomio}(2) 
\[\prod_{i=1}^{r}\binom{p^{S+e}|v_i|}{p^{S+e}v_i}c_{a(i)}^{d_{e}(i)} \cdot  \prod_{i=r+1}^{t}\binom{p^{S+e}|\langle \rho_i \rangle_{S+e}|}{p^{S+e}\langle \rho_i \rangle_{S+e}}c_{a(i)}^{b_{e}(i)}
\]
appears as the coefficient of the monomial 
\[(x^{p^{S+e}  \widehat{\E}_1 v_1}  \cdots  x^{p^{S+e}\widehat{\E}_r v_r})\cdot (  x^{p^{S+e}\widehat{\E}_{r+1} \langle \rho_{r+1} \rangle_{S+e}}\cdots   x^{p^{S+e}\widehat{\E}_t \langle \rho_t \rangle_{S+e} })
\]
in the polynomial $g$ given by the product 
$\prod_{i=1}^{r} f_i^{p^{S+e}|v_i|} \cdot \prod_{i =r+1}^t  f_i^{p^{S+e}|\langle \rho_i \rangle_{S+e}|}$. By the definitions of  $S$ and $S_i$, we conclude from Lucas's Lemma that $\binom{p^{S+e}|v_i|}{p^{S+e}v_i} \not \equiv 0$ mod $p$ for each $v_i$. Since $S_i=\infty$, 
for $ i \in \{r+1,\ldots, t\}$, then $\binom{p^{S+e}|\langle \rho_i \rangle_{S+e}|}{p^{S+e}\langle \rho_i \rangle_{S+e}} \not \equiv 0$ mod $p$. In conclusion one can set $Q(e)=p^{S+e} |v| +p^{S+e}|\langle( \rho_{r+1},\ldots, \rho_t) \rangle_{S+e}|$ to verify that 
$g \in  \fa^{Q(e)}$ but $g \notin \m^{[p^{S+e}]}$. In this way $Q(e)\leq \nu_{\fa}(p^{S+e})$. Dividing by $p^{S+e}$ and letting $e \rightarrow \infty$ yields finally  
\[\sum_{i=1}^{r}\left(|\langle \rho_i \rangle_{S_i}|+  \frac{1}{p^{S_i}}\right) + \sum_{i=r+1}^t|\rho_i| \leq \fpt(\fa).
\]
\end{proof}

\begin{example}\label{Example fptIdeal}
Take $R=\FF_p[x,y,z]$ and let $\fa$ be the ideal generated by $\Bf=(x^2+xy^2,yz^3)$. Then $\SP_{\Bf}$ is the splitting polytope of Example \ref{Ex:SplitPol} with $a=b=2$, and $c=3$ which has a unique maximal point $\rho=\left(\frac{1}{3} ,\frac{1}{3}, \frac{1}{3} \right)$. 
For $p=2$, 
\[
\frac{1}{3}=\sum_{e\geq 1} \frac{\alpha^{(e)}}{2^{e}},\quad \text{where }\, \alpha^{(e)}=0 \, \text{ for }\,  e \text{ odd and }\,  \alpha^{(e)}=1 \text{ for even } e,
\]
thus $S_1=1,\; S_2=\infty$ and $\langle \rho_1 \rangle_1=(0,0)$. By Theorem~\ref{Thm:F-pure-threshold-polytope}(2)  $\fpt(\fa)\geq |\langle \rho_1 \rangle_1|+\frac{1}{2}+\frac{1}{3}=\frac{5}{6}$. 
If $p=3$, then $\frac{1}{3}=\sum_{e\geq 2} \frac{2}{3^{e}}$, again $S_1=1,\; S_2=\infty$ and $\fpt(\fa)\geq \frac{2}{3}$.

Assume now that $p=6w+1$ for some $w\geq 1$, then $\frac{1}{3}=\sum_{e\geq 1} \frac{2w}{p^{e}}$ and $S_1=S_2=\infty$, hence $\fpt(\fa)= | \rho|=1$ by Theorem~\ref{Thm:F-pure-threshold-polytope}(1). Finally if $p=6w+5$ for some $w\geq 0$, then $\frac{1}{3}=\sum_{e\geq 1} \frac{\alpha^{(e)}}{p^{e}}$, where $\alpha^{(e)}=2w+1$ for $e$ odd and  $\alpha^{(e)}=4w+3$ for $e$ even. In this case $S_1=1,\;S_2=\infty$ and $\langle \rho_1 \rangle_1=\left(\frac{2w+1}{p},\frac{2w+1}{p}\right)$, giving $\fpt(\fa)\geq \frac{4w+2}{p}+\frac{1}{p}+\frac{1}{3}=1-\frac{1}{3p}$. Summarizing,
$$
	\begin{cases}
		 \fpt(\fa) \geq \frac{5}{6}& \text{if}\quad p=2\\
		\fpt(\fa) \geq \frac{2}{3}& \text{if}\quad p=3\\
		\fpt(\fa)=1& \text{if}\quad p\equiv 1 \mod 6 \\
		\fpt(\fa)\geq 1-\frac{1}{3p}& \text{if}\quad p\equiv 5 \mod 6.
	\end{cases}
	$$ 
\end{example}
\begin{example}\label{Ex:fptequallct}
	Following Example \ref{Ex:NewtPolToMaxFace}, consider the ideal $\fa$ generated by
	\[\Bf=(y^{5}(x+y^m)^4,x^n),\quad \text{with}\quad m\geq 1\quad \text{and}\quad n > 4+5/m.
	\]
	The unique maximal point of $\SP_{\Bf}$ is $\{(1/5,0,0,0,0,1/5n)\}$. It is not difficult to show that for any prime $p$, $S_1$ and $S_2$ in Theorem \ref{Thm:F-pure-threshold-polytope} are both infinite. Therefore $\fpt(\fa)=\frac{1}{5}+\frac{1}{5n}$ which coincides with  $\lct((y^{5}(x+y^m)^4,x^n))$.
\end{example}

\subsection{Relations with the log canonical threshold}\label{SSec:Relations}
Let $\mathfrak{a}$ be an ideal in  $\QQ[x_1,\ldots, x_m]$ and consider a log resolution of $\fa$ defined over $\QQ$. We mean by this, a proper birational map $\pi: Y \to \mathbb{A}^m$, with $Y$ smooth, and such that $\fa\cdot \mathcal{O}_{Y}=\mathcal{O}_{Y}(-D)$ is invertible and $\operatorname{Exc}(\pi) \cup \Supp(D)$ is a simple normal crossing divisor. Denote by $K_\pi$ the relative canonical divisor on $Y$ and set $D$ for the effective divisor given by $\pi^{-1}(\fa)$, then the \textit{log canonical threshold of} $\fa$ (at the origin) is given by
\[\lct(\fa) = \sup\{\alpha \in \RR_{>0} \mid  \lceil K_\pi -\alpha D\rceil \,\text{ is effective}\},
\]
where $\lceil H\rceil$ denotes the round up divisor of $H$. Given a prime number $p$, one may consider the reduction $\fa_p=\fa\cdot\FF_p[x_1,\ldots, x_m]_{(x_1,\ldots,x_m)}$ of $\fa$ mod $p$, and then it is very natural to ask about the relationship among the corresponding thresholds. Following Hara and Yosida~\cite{HaYo}, Musta\c{t}\v{a}, Tagaki, and Watanabe \cite[Theorem 3.3]{MTW} observe that for $p$ large enough, one has
\begin{equation}\label{Eq:ResultsMTW}
	\fpt(\fa_p)\leq \lct(\fa) \quad \text{and}\quad \lim_{p\to \infty} \fpt(\fa_p)=\lct(\fa).
\end{equation} 
The authors also posed the following conjecture which still represents an interesting open challenge. For historical comments and some references of the proven cases consult the survey \cite{BFS2013}.
\begin{conjecture}[{\cite[Conjecture 3.6]{MTW}}]\label{Conj: fptandlct}
	There are infinitely many primes $p$ for which $\fpt(\fa_p) = \lct(\fa)$.
\end{conjecture}
\begin{example}\label{Ex:ContExFptIdeal}
	Following Example \ref{Example fptIdeal}, let $\fa$ be the ideal generated by $\Bf=(x^2+xy^2,yz^3)$. It is not difficult to show that $\lct(\fa)=1$ which coincides with $\fpt(\fa)$ for the primes $p$ verifying $p \equiv 1 \mod 6$.
\end{example}

\begin{notation}\label{Notation}
	\begin{enumerate}
		\item For a given ideal $\fa$, we denote by $P_\fa$ the set of prime numbers for which  $\N_{\fa}=\N_{\fa_p}$.
		\item For a given vector $\rho=(\rho_1, \ldots, \rho_t)\in \RR^N$ having at least one index $i$ with $|\rho_i|>1$, assume (reordering if necessary) that for $k\in \{1,\ldots,r\}$, $|\rho_k|>1$. Denote by $P_\rho^0$ the set of prime numbers satisfying the condition 
		\begin{equation}\label{Eq:Condition0}
			\rho_{k,1}^{(1)}+ \cdots + \rho_{k,l_{k}}^{(1)} \geq p
		\end{equation} 
		for every $k\in\{1,\ldots,r\}$, c.f. Lemma~\ref{lema:lemma8}.
		\item For $\rho=(\rho_1, \ldots, \rho_t)\in \RR^N$ we denote by $P_\rho^\infty$ the set of prime numbers satisfying the condition that $\rho_k$ adds without carrying for every $k\in\{1,\ldots,t\}$.
	\end{enumerate}
\end{notation}
Note that condition~\ref{Eq:Condition0} holds for a sufficiently large prime number $p$, more precisely it holds for $p\geq \frac{l_k}{|\rho_k|-1}$. 

\begin{theorem}\label{Thm:fpt-equals-lct}
	Let $\fa \subseteq \m$ be a nonzero ideal with term ideal $\fa^\circ$, and take a list $\Bf=(f_1,\ldots,f_t)$  of minimal generators of $\fa$. Suppose that $\SP_{\Bf}$ has a unique maximal point $\rho=(\rho_1, \ldots, \rho_t)$ with $\rho_i \in \RR^{l_i}$. 
	\begin{enumerate}
		\item Assume that $\rho$ verifies the following property: whenever there is one $k_0\in\{1,\ldots,t\}$ with $|\rho_{k_0}|>1$, then $|\rho_k|>1$ for every $k\in\{1,\ldots,t\}$. If $\lct(\fa^\circ)>t$, then $\fpt(\fa_p)=\lct(\fa)=t$ for big enough primes in  $P_\rho^0\cap P_\fa$. 
		
		\item Assume that $\rho$ verifies the following property: whenever there is one $k_0\in\{1,\ldots,t\}$ with $|\rho_{k_0}|\leq 1$, then $|\rho_k|\leq 1$ for every $k\in\{1,\ldots,t\}$. Assume moreover that $P_\rho^\infty$ is an infinite set. If $\lct(\fa^\circ)\leq t$, then $\fpt(\fa_p)=\lct(\fa)$ for big enough primes in  $P_\rho^\infty\cap P_\fa$.
	\end{enumerate}	
\end{theorem}
\begin{proof}
	Since $\fa_p^\circ$ is a monomial ideal it follows from Proposition~\ref{Prop:Properties ftt ideals}(2) and Example~\ref{Ex:Fthrsholds} that 
	$|\rho|=\fpt(\fa_p^\circ)=\lct(\fa_p^\circ)$. To prove the first part of the theorem, note that for $p\in P_\fa$, the equality $\N_{\fa}=\N_{\fa_p}$ guarantees that the log canonical threshold of the corresponding term ideals are equal, yielding
	\begin{equation}\label{Eq:EqualityFPTs}
		|\rho|=\fpt(\fa_p^\circ)=\lct(\fa_p^\circ)=\lct(\fa^\circ).
	\end{equation}
	From the hypothesis $\lct(\fa^\circ)>t$, there should be at least one $k_0\in\{1,\ldots,t\}$ with $|\rho_{k_0}|>1$, which also implies that $|\rho_k|>1$ for every $k\in\{1,\ldots,t\}$. Now we use Lemma~\ref{lema:lemma8}(2) to conclude that $\rho_{k,1}^{(1)}+ \cdots + \rho_{k,l_{k}}^{(1)} \geq p$, for $p\in P_\rho^0\cap P_\fa$, meaning that $S_{k}=0$ (recall Definition~\ref{Def:S-i}). Summarizing, we have shown that $S_1=\cdots=S_t=0$, and then Theorem~\ref{Thm:F-pure-threshold-polytope}(2) gives
	\begin{equation*}\label{Eq:InequalityFPT} 
		t\leq \fpt(\fa_p).
	\end{equation*}
	In fact this inequality is an equality thanks to property $\fpt(\fa_p) \leq t$ from Proposition~\ref{Prop:Properties ftt ideals}(1). Finally note that for $p$ big enough in $P_\rho^0\cap P_\fa$ we have, see \eqref{Eq:ResultsMTW}, 
		\[t=\fpt(\fa_p)\leq \lct(\fa).\]
	The conclusion then follows from the well known inequality $\lct(\fa) \leq t$.

	For the second part of the theorem we observe that  $\lct(\fa^\circ)\leq t$, implies that $|\rho_1|+\cdots+|\rho_t|\leq t$. Then there should be an index $k_0\in\{1,\ldots,t\}$ for which $|\rho_{k_0}|\leq 1$, which in turn implies $|\rho_k|\leq 1$ for every $k\in\{1,\ldots,t\}$. Lemma~\ref{lema:lemma8}(1) implies that $\rho_{k}$ adds without carrying for infinitely many $p$. Now we use the fact that 
	$P_\rho^\infty$ is an infinite set, in particular this is saying that the intersection of the sets of primes in which each $\rho_{k}$ adds without carrying is non-empty and infinite. We conclude that $S_1=\cdots=S_t=\infty$ for any prime in $P_\rho^\infty$. This time we use Theorem~\ref{Thm:F-pure-threshold-polytope}(1) to conclude $\fpt(\fa_p)=\fpt(\fa_p^\circ)$. Note that for $p$ big enough in $P_\rho^\infty\cap P_\fa$ we have $\fpt(\fa_p)\leq \lct(\fa)$ from the left hand side of \eqref{Eq:ResultsMTW}, implying 
	$$\fpt(\fa_p)\leq \lct(\fa) \leq \lct(\fa^\circ).$$ Finally we may complete the chain of inequalities as $$\fpt(\fa_p^\circ)=\fpt(\fa_p)\leq \lct(\fa) \leq \lct(\fa^\circ)=\fpt(\fa_p^\circ),$$ finishing this way the proof.  
\end{proof}
\begin{example}\label{Ex:CondPrimes}
	Let $\fa$ be the ideal of $\QQ[x_1,\ldots,x_6]$ generated by $\Bf=(x_1^{a}+x_2^{b}+x_3^{c},x_4^{a}+x_5^{b}+x_6^{c})$. Then
	$\N_{\fa}$ has just one compact face with equation 
	\[\frac{x_1}{a}+\frac{x_2}{b}+\frac{x_3}{c}+\frac{x_4}{a}+\frac{x_5}{b}+\frac{x_6}{c}=1,
	\]
	and thanks to Proposition~\ref{Prop:SplPolvsNewPol}, $\SP_{\Bf}$ has a unique maximal point with coordinates $\rho=(\frac{1}{a},\frac{1}{b},\frac{1}{c},\frac{1}{a},\frac{1}{b},\frac{1}{c})$. In this case $|\rho_1|=|\rho_2|=\frac{1}{a}+\frac{1}{b}+\frac{1}{c}$ and Example~\ref{Ex:Fthrsholds} gives 
	\[\lct(\fa^\circ)=2\left(\frac{1}{a}+\frac{1}{b}+\frac{1}{c}\right).
	\]
	If we assume that $\lct(\fa^\circ)>2$, then the first part of Theorem~\ref{Thm:fpt-equals-lct} gives $\fpt(\fa_p)=\lct(\fa)=2$ for any prime $p$ large enough (in this case $P_\rho^0\cap P_\fa$ consists of the primes $p>3$).
	
	Now assume that $\lct(\fa^\circ)\leq 2$, then $|\rho_1|=|\rho_2|\leq 1$. The second part of Theorem~\ref{Thm:fpt-equals-lct} yields 
	$\fpt(\fa_p)=\lct(\fa)=2\left(\frac{1}{a}+\frac{1}{b}+\frac{1}{c}\right),$ for big enough primes in $P_\rho^\infty\cap P_\fa$. From Example~\ref{Ex:pExpansion} one member of $P_\rho^\infty$ is a prime $p$ verifying $p \equiv 1 \mod abc$, provided that $ab+bc+ac\leq abc$.
	
\end{example}

Some approaches to the algorithmic calculus of $F$-pure thresholds are possible. In this direction, the only software that we are aware of, is the  $\mathrm{Macaulay2}$ package \textit{FrobeniusThresholds} \cite{MacFrobThre}, designed to estimate and calculate $F$-pure thresholds in several cases. This package incorporates some of the techniques for principal ideals given in \cite{HDbinomial,HDdiagonal,HDpolynomials}. 
Another algorithmic approach to the explicit calculation of $F$-pure thresholds (and log canonical thresholds) is considered by Shibuta and Takagi \cite{ShTa}, where they proved that $\lct(\fa)$ is computable by linear programming for several classes of binomial ideals. The following example is taken from them. The general case of binomial ideals is treated in \cite{BlaEn}.

\begin{example}[Space Monomial Curves]\label{Exam:SpaceMonCurves}
	Take $R=\FF_p[x,y,z]$ and let $\fa$ be the ideal generated by $\Bf=(x^a-y^b,z^c-x^dy^e)$, with $a,b,c\geq 1$, $d\geq e\geq 0$ and  $ae+bd\geq ab$. Then $\SP_{\Bf}=\{(\gamma_1,\gamma_2,\gamma_3,\gamma_4)\in \RR_{\geq 0}^4 \mid a\gamma_1+d\gamma_4 \leq 1, b\gamma_2+e\gamma_4 \leq 1, \text{ and } c\gamma_3 \leq 1 \}$, while $\N_{\fa}$ has just one compact face which is the convex hull of $(a,0,0)$, $(0,b,0)$ and $(0,0,c)$. Proposition \ref{Prop:maximal face} guarantees that the unique maximal face of $\SP_{\Bf}$ is $\rho=(\frac{1}{a},\frac{1}{b},\frac{1}{c},0)$. For instance, if 
	$a+b\leq ab$ then the second part of Theorem~\ref{Thm:fpt-equals-lct} implies that $\fpt(\fa_p)=\lct(\fa)=|\rho|$ for large enough primes in $P_\rho^\infty\cap P_\fa$. Again Example~\ref{Ex:pExpansion} provides an element of $P_\rho^\infty$.
\end{example}

\section{Mixed Generalized Test Ideals and $F$-Volumes}

In this section we present a final application of our construction of the splitting polytope of a polynomial mapping. We first introduce the
definition of the $F$-volume of a collection of ideals $\mathfrak{a}_1, \ldots, \mathfrak{a}_t$ and some of its properties. Then we use the proof of Theorem \ref{Thm:F-pure-threshold-polytope} for the construction of a lower bound for this volume.

Test ideals were introduced by Hochster and Huneke \cite{HoHu90,HoHu94} as the characteristic $p$ analogues of multiplier ideals. We follow the  description of Bickle, Musta{\c{t}}{\u{a}}, and Smith \cite{BMS-MMJ} and consider $R$ a polynomial ring over a field $\mathbb{k}$ of prime characteristic $p$. Assume further that $R$ is $F$-finite, meaning that the Frobenius map is finite. 
\begin{definition}
	\begin{enumerate}
		\item Given an ideal $\mathfrak{b} \subseteq R$, the $p^{e}$-th \textit{Frobenius root} of $\mathfrak{b}$, denoted $\mathfrak{b}^{[1/p^{e}]}$, is the smallest ideal $\fa$ of $R$ such that $\mathfrak{b} \subseteq \fa^{[p^{e}]}$. 
		\item Given a collection $\mathfrak{a}_1, \ldots, \mathfrak{a}_t$ of ideals in $R$ and $c=(c_1,\ldots,c_t) \in \mathbb{R}^{t}_{\geq 0}$, the \textit{mixed generalized test ideal} with exponents $c_1,  \ldots, c_t$ is defined as
		\[
		\tau(\mathfrak{a}_1^{c_1}\cdots \fa_{t}^{c_t})= \bigcup_{e>0} (\mathfrak{a}_1^{\lceil c_1p^{e} \rceil}\cdots \mathfrak{a}_t^{\lceil c_tp^{e} \rceil})^{[1/p^{e}]}.
		\]
	\end{enumerate} 
\end{definition} 
Since $R$ is Noetherian, the ascending family of ideals $\set{(\mathfrak{a_1}^{\lceil c_1p^{e} \rceil}\cdots \mathfrak{a_t}^{\lceil c_tp^{e} \rceil})^{[1/p^{e}]}}_{e>0}$ stabilizes, therefore $\tau(\mathfrak{a}_1^{c_1}\cdots \fa_{t}^{c_t})= (\mathfrak{a_1}^{\lceil c_1p^{e} \rceil}\cdots \mathfrak{a_t}^{\lceil c_tp^{e} \rceil})^{[1/p^{e}]}$ for $e$ big enough. Note that for a given collection of ideals $\mathfrak{a}_1, \ldots, \mathfrak{a}_t$ and $c\in \RR_{\geq 0}^t$, there always exist points $d\in \RR_{\geq 0}^t$, other than $c$ such that $\tau(\mathfrak{a}_1^{c_1}\cdots \fa_{t}^{c_t})=\tau(\mathfrak{a}_1^{d_1}\cdots \fa_{t}^{d_t})$. The subset of $\RR_{\geq 0}^t$ consisting of all such $d$ is called the \textit{constancy region} of $\tau(\mathfrak{a}_1^{c_1}\cdots \fa_{t}^{c_t})$, see for instance \cite{HaYo,TakagiAdj,BMS-MMJ,Pe2013}. The measure of a related set in $\RR_{\geq 0}^t$ was introduced under the name of $F$-volumes by  Badilla-C\'{e}spedes, N\'{u}\~{n}ez-Betancourt and Rodr\'{i}guez-Villalobos \cite{BCNBRV}.

\begin{definition}\label{def1}
	Consider a collection $\mathfrak{a}_1, \ldots, \mathfrak{a}_t, \fb$ of ideals in $R$ and assume that for each $i\in\{1,\ldots,t\}$, $\fa_i\subseteq \sqrt{\fb}$. Given $e\in \mathbb{Z}_{\geq 0}$ we define
	\[
	\V_{\mathfrak{a}_1, \ldots, \mathfrak{a}_t}^{\fb}(p^e)=\{(n_1,\ldots, n_t)\in\mathbb{Z}_{\geq 0}^t\ \mid  
	\fa_1^{n_1}\cdots \fa_t^{n_t}\not \subseteq \fb^{[p^e]}\}.
	\]
\end{definition}
\begin{theorem}[{\cite[Theorem $2.13$]{BCNBRV}}]\label{ThmLimitExistsSec}
	The limit $$\lim_{e\to \infty}\frac{1}{{p^{et}}} \cdot \Card ( \V_{\mathfrak{a}_1, \ldots, \mathfrak{a}_t}^{\fb}(p^{e}))$$
	converges.
\end{theorem}
The previous limit is called the $F$-volume of $\mathfrak{a}_1, \ldots, \mathfrak{a}_t$ with respect to $\fb$ and it is denoted by $\Vol_F^\fb(\mathfrak{a}_1, \ldots, \mathfrak{a}_t)$. The $F$-volume measures the sum of the volumes of the constancy regions where $\tau(\fa^{n_1}_1 \cdots \fa^{n_t}_t)\not\subseteq \fb$. It is known~\cite{BCNBRV} that  the $F$-volume of $\mathfrak{a}_1, \ldots, \mathfrak{a}_t$ with respect to $\m=(x_1, \ldots, x_m)$ is bounded above by the minimum between $\fpt(\fa_1)\cdots\fpt(\fa_t)$ and the volume of the pyramid in $\RR_{\geq 0}^t$ bounded by the coordinate hyperplanes and the hyperplane defined by $x_1+\ldots+x_t=\fpt(\fa_1+\ldots +\fa_t)$. However, no trivial lower bounds are known for the $F$-volume  of $\mathfrak{a}_1, \ldots, \mathfrak{a}_t$ with respect to $\m$ so we provide some bounds in terms of our splitting polytopes $\SP_{\Bf}$. First we start with an easy calculation.

\begin{proposition} \label{prop-volume-term-ideal}
Let $\Bf=(f_1,\ldots,f_t)$ be a polynomial mapping of elements in $\m$, and take $\fa_1^\circ,\ldots,\fa_t^\circ$ the collection formed by the term ideals of the principal ideals $(f_1),\ldots,(f_t)$. Assume that a point $\gamma \in \SP_{\Bf}$ is written as $\gamma=(\gamma_1, \ldots, \gamma_t)$ with $\gamma_i \in \RR^{l_i}$, then
$$\max_{\gamma \in \SP_{\Bf}}  |\gamma_1| \cdots  |\gamma_t|  \leq \Vol_F^\m(\fa_1^\circ,\ldots,\fa_t^\circ).$$
\end{proposition}
\begin{proof} 
By definition  $\gamma\in \SP_{\Bf}$ means that $\E_{\Bf} \langle \gamma \rangle_e \prec \E_{\Bf} \gamma \preceq  \textbf{1}_m$, implying  $x^{p^{e}\tE_1 \langle \gamma_1 \rangle_e}  \cdots x^{p^{e}\tE_t \langle \gamma_t \rangle_e} \in (\fa_1^\circ)^{p^{e}|\langle \gamma_1 \rangle_e|} \cdots (\fa_t^\circ)^{p^{e}| \langle \gamma_t \rangle_e|} \setminus \m^{[p^{e}]}$. Therefore   
$(p^{e}|\langle \gamma_1 \rangle_e|, \ldots, p^{e} | \langle \gamma_t \rangle_e|)$ is a point of $\V_{\fa_1^\circ,\ldots,\fa_t^\circ}^{\m}(p^{e})$, and we have the bound
 $$p^{et}\prod_{i=1}^t | \langle \gamma_i \rangle_e| \leq \Card \left( \V_{\fa_1^\circ,\ldots,\fa_t^\circ}^{\m}(p^{e}) \right).$$ 
Dividing by $p^{et}$ and letting $e \rightarrow \infty$ yields $\prod_{i=1}^t | \gamma_i| \leq \Vol_F^{\m}(\fa_1^\circ,\ldots,\fa_t^\circ)$, 
from which the conclusion easily follows.
\end{proof}

\begin{theorem} \label{Thm:volume-polytope}
	Consider a polynomial mapping $\boldsymbol{f}=(f_1,\ldots, f_t)$ of elements in $\m$. Suppose that $\SP_{\Bf}$ has a unique maximal point $\rho=(\rho_1, \ldots, \rho_t)$ with $\rho_i \in \RR^{l_i}$, and set for $i\in\{1,\ldots,t\}$,
	$$S_i=\sup \left\{\ell\in\mathbb{Z}_{\geq 0} \mid \sum_{j=1}^{l_i} \rho_{i,j}^{(e)}  \leq p-1\quad\text{for every}\quad 0 \leq e \leq \ell  \right\}.$$ 
	If $I$ denotes the set of indices for which $S_i$ is finite,  then
	$$ \prod_{i\in I} \left(|\langle \rho_i \rangle_{S_i}|+\frac{1}{p^{S_i}} \right)  \prod_{i\notin I} |\rho_i|\leq \Vol_F^{\m}((f_1),\ldots,(f_t)).$$ 
\end{theorem}
\begin{proof}
	Following the proof of Theorem~\ref{Thm:F-pure-threshold-polytope}(1) we note that for $I=\emptyset$, the polynomial $$g=\prod_{i=1}^{t} f_i^{p^{e}|\langle \rho_i \rangle_e|} \notin \m^{[p^{e}]}.$$
	This implies that $(p^{e}| \langle \rho_1 \rangle_e|, \ldots, p^{e} |\langle \rho_t \rangle_e |)$ belongs to $ \V_{(f_1),\ldots,(f_t)}^{\m}(p^{e})$, and also $$p^{et}\prod_{i=1}^t | \langle \rho_i \rangle_e| \leq \Card \left( \V_{(f_1),\ldots,(f_t)}^{\m}(p^{e}) \right).$$ Dividing by $p^{et}$ and letting $e \rightarrow \infty$ concludes the proof of this case.
	
	We follow the same strategy when $I \neq \emptyset$, but using this time the polynomial 
	$$g=\prod_{i=1}^{r} f_i^{p^{S+e}|v_i|} \cdot \prod_{i =r+1}^t  f_i^{p^{S+e}|\langle \rho_i \rangle_{S+e}|},$$
	from the second part of the proof of Theorem \ref{Thm:F-pure-threshold-polytope}. Therefore 
	 $$p^{(S+e)t} \left(\prod_{i=1}^r | v_i | \cdot \prod_{i=r+1}^{t} |\langle \rho_i \rangle_{S+e}| \right)  \leq \Card \left( \V_{(f_1),\ldots,(f_t)}^{\m}(p^{S+e}) \right).$$ 
	The conclusion is achieved as before, dividing by $p^{(S+e)t}$ and taking the limit.
\end{proof}
\begin{example}
	Take $R=\FF_p[x,y,z]$ and $\Bf=(x^2+xy^2,yz^3)$ as in Example \ref{Example fptIdeal}. Then 
$$
	\begin{cases}
		 \Vol_F^{\m}((x^2+xy^2),(yz^3)) \geq \frac{1}{6}& \text{if}\quad p=2\\
		\Vol_F^{\m}((x^2+xy^2),(yz^3)) \geq \frac{1}{9}& \text{if}\quad p=3\\
		\Vol_F^{\m}((x^2+xy^2),(yz^3))\geq \frac{2}{9}& \text{if}\quad p\equiv 1 \mod 6 \\
		\Vol_F^{\m}((x^2+xy^2),(yz^3))\geq \frac{2}{9}-\frac{1}{9p}& \text{if}\quad p\equiv 5 \mod 6.
	\end{cases}
	$$ 
\end{example}
Of course the inequality of Theorem \ref{Thm:volume-polytope} can be strict as we will show next. This demonstrates the need to continue looking for bounds for the $F$-volume.
\begin{example}\label{equality-fail1}
Let $R=\FF_2[x,y]$ and consider the polynomial mapping $\Bf=(x,x+y^2)$.Then $\Vol_F^{\m}((x),(x+y^2))=3/4$ \cite[Example $2.15$]{BCNBRV}. 
On the other hand, $\SP_{\Bf}$ is exactly $[0,1]\times \left[0,\frac{1}{2}\right]$ with maximal point $(\rho_1,\rho_2)=(1,1/2)$. Since $S_1=S_2=\infty$, we obtain $|\rho_1| \cdot |\rho_2| =1/2$. 
\end{example}
		
\section*{Acknowledgments}
We thank  Manuel Gonz\'{a}lez Villa for inspiring conversations. We thank Luis N\'{u}{\~n}ez-Betancourt for helpful comments and suggestions. We also
thank to an anonymous referee for helpful comments.

\providecommand{\bysame}{\leavevmode\hbox to3em{\hrulefill}\thinspace}
\bigskip

\bibliographystyle{alpha}

\end{document}